\documentclass[reqno,11pt]{amsart}
\usepackage{amssymb,amsmath,mathrsfs}
\usepackage{enumerate,tikz}
\usetikzlibrary{shapes.geometric}

\numberwithin{equation}{section}
\newtheorem{theorem}{Theorem}
\newtheorem{lemma}[theorem]{Lemma}
\newtheorem{corollary}[theorem]{Corollary}
\newtheorem{proposition}[theorem]{Proposition}

\theoremstyle{definition}

\def\p{\hat p}

\title[Colored partitions of a convex polygon by noncrossing diagonals]{
	Colored partitions of a convex polygon \\ by noncrossing diagonals}
\author{Daniel Birmajer}
\address{Department of Mathematics\\ Nazareth College\\ 4245 East Ave.\\ Rochester, NY 14618}
\author{Juan B. Gil}
\address{Penn State Altoona\\ 3000 Ivyside Park\\ Altoona, PA 16601}
\author{Michael D. Weiner}

\begin{document}
\maketitle

\begin{abstract}
For any positive integers $a$ and $b$, we enumerate all colored partitions made by noncrossing diagonals of a convex polygon into polygons whose number of sides is congruent to $b$ modulo $a$. For the number of such partitions made by a fixed number of diagonals, we give both a recurrence relation and an explicit representation in terms of partial Bell polynomials. We use basic properties of these polynomials to efficiently incorporate restrictions on the type of polygons allowed in the partitions. 
\end{abstract}

\section{Introduction}
In this paper we consider the problem of counting the number of colored partitions of a convex polygon, allowing restrictions on the shape of the underlying parts. More precisely, given $a, b\in\mathbb{N}$, and $\mathbf{c}=(c_1,c_2,\dots)$ with $c_j\in\mathbb{N}_0$, we are interested in the set
\smallskip
\begin{quote}
$\mathscr{P}^{\mathbf{c}}_{a,b}(r,k)$ of colored partitions of a convex $(r+2)$-gon made by $k-1$ noncrossing diagonals into polygons whose number of sides is congruent to $b$ modulo $a$, and each $(aj+b)$-gon may be colored in $c_j$ ways.
\end{quote}
This is a generalization of a classical problem solved by Cayley in his paper ``On the partition of a polygon'' from 1891, see \cite{Cayley}. At the beginning of his paper, Cayley refers to Kirkman \cite{Kirkman} and to Taylor and Rowe \cite{TaylorRowe} for a historical account on the problem.

\smallskip
Observe that for any partition in $\mathscr{P}^{\mathbf{c}}_{a,b}(r,k)$ consisting of $(aj_{\ell}+b)$-gons, $\ell=1,\dots,k$, we must have
\begin{equation*}
  r+2 = (aj_1+b)+\dots+(aj_k+b) - 2(k-1).
\end{equation*}
Thus, letting $n=j_1+\dots+j_k$, we get $r=an+(b-2)k$ and so
\begin{equation*}
|\mathscr{P}^{\mathbf{c}}_{a,b}(r,k)|= 0 \text{ unless } r=an+(b-2)k \text{ for some } n\in\mathbb{N}.
\end{equation*}
Our first result concerns a recurrence relation for the cardinality of $\mathscr{P}^{\mathbf{c}}_{a,b}(r,k)$.
\begin{theorem}\label{thm:qrecursion}
Let $a, b\in\mathbb{N}$, $\mathbf{c}=(c_1,c_2,\dots)$, and let $\p(n,k)=|\mathscr{P}^{\mathbf{c}}_{a,b}(an+(b-2)k,k)|$. Then, for $k>1$, we have
\begin{equation}\label{eqn:qrecursion}
 \frac{2(k-1)}{r+2}\p(n,k) = \sum_{\ell=1}^{k-1}\sum_{m=\ell}^{n-1} \p(n-m,k-\ell)\p(m,\ell).
\end{equation}
\end{theorem}

\medskip
A combinatorial proof of this theorem is given in Section~\ref{sec:qrecursionProof}. Incidentally, the convolution formula \eqref{eqn:qrecursion} leads to the well-known partial Bell polynomials. In fact, the quantity $\p(n,k)$ admits the following explicit representation.
\begin{theorem}\label{thm:qBell}
\begin{equation}\label{eqn:qBell}
\p(n,k) =\frac1{an+(b-1)k+1}\binom{an+(b-1)k+1}{k}\frac{k!}{n!}B_{n,k}(1!c_1,2!c_2,\dots).
\end{equation}
\end{theorem}
Here $B_{n,k}=B_{n,k}(x_1,x_2,\dots)$ denotes the $(n,k)$-th partial Bell polynomial in the variables $x_1,x_2,\dots,x_{n-k+1}$. For the definition and main properties of these polynomials, we refer to the book by Comtet \cite{Comtet}.

\medskip
The proof of the Theorem~\ref{thm:qBell}, given in Section~\ref{sec:pBellProof}, relies on a resourceful convolution result for partial Bell polynomials proved by the authors in \cite{BGW12}. As we will discuss in later sections, the representation \eqref{eqn:qBell} leads to a unifying approach and a simpler treatment for several polygon partition problems. In particular, notice that if $a=1$ and $b=2$, the number of colored partitions of a convex $(r+2)$-gon into $(j+2)$-gons, $j=1,\dots,r$, made by $k-1$ noncrossing diagonals, is given by
\begin{align*} 
 \p(r,k) &=\frac1{r+k+1}\binom{r+k+1}{k}\frac{k!}{r!}B_{r,k}(1!c_1,2!c_2,\dots) \\
 &=\frac1{r+1}\binom{r+k}{k}\frac{k!}{r!}B_{r,k}(1!c_1,2!c_2,\dots).
\end{align*}
If $c_j=1$ for every $j$ (i.e. only one color for each partition), then $\frac{k!}{r!}B_{r,k}(1!,2!,\dots) = \binom{r-1}{k-1}$ and we recover Cayley's formula 
\begin{equation*} 
   \p(r,k) = \frac{1}{r+1} \binom{r+k}{k}\binom{r-1}{k-1}, \;\; r, k\in\mathbb{N}.
\end{equation*}
Alternative proofs of this formula can be found in \cite{Stanley}, \cite{PS2000}, and most recently in \cite{Gaiffi}.

In Section~\ref{sec:applications}, we consider more examples that illustrate the versatility of formula \eqref{eqn:qBell}. More specifically, we examine several scenarios where only specific types of polygons are included or excluded in the partitions. Some of the sequences discussed in this section were considered by Smiley \cite{Smiley} using reversion of generating functions. 

In the last section, we enumerate a subset of $\mathscr{P}^{\mathbf{c}}_{a,b}(r,k)$ in which partitions of a given convex polygon are required to contain a $(d+1)$-gon over a fixed side of the polygon. This related problem is solved by means of a multifold convolution formula obtained via corresponding convolutions of partial Bell polynomials.

Overall, we find that partial Bell polynomials provide a very efficient tool for the study of combinatorial problems that involve recurrence relations of convolution type. This includes similar counting problems, notably in the context of rooted trees and Dyck paths.

\section{Proof of Theorem~\ref{thm:qrecursion}} 
\label{sec:qrecursionProof}
In this section we give a combinatorial proof for the recursion given in Theorem~\ref{thm:qrecursion}.
Recall that $\mathscr{P}^{\mathbf{c}}_{a,b}(r,k)$ is the set of colored partitions of a convex $(r+2)$-gon made by $k-1$ diagonals into polygons whose number of sides is congruent to $b$ modulo $a$, where each $(aj+b)$-gon may be colored in $c_j$ ways.

Let $k>1$ and consider the set
\smallskip
\begin{quote}
$\mathscr{R}^{\mathbf{c}}_{a,b}(r,k)$ of colored partitions in $\mathscr{P}^{\mathbf{c}}_{a,b}(r,k)$ with the additional structure of a specified diagonal with one of its vertices flagged.
\end{quote}
\smallskip
Since every partition by $k-1$ diagonals has $2(k-1)$ possibilities to flag a vertex, we have
\begin{equation} \label{eqn:split1}
  \big|\mathscr{R}^{\mathbf{c}}_{a,b}(r,k)\big|=2(k-1)\big|\mathscr{P}^{\mathbf{c}}_{a,b}(r,k)\big|.
\end{equation}
On the other hand, for any $(r+2)$-gon we can write $\mathscr{R}^{\mathbf{c}}_{a,b}(r,k)$ as a disjoint union
\[ \mathscr{R}^{\mathbf{c}}_{a,b}(r,k) = \bigsqcup_{v} R_v(r,k) \]
over all vertices $v$ of the polygon, where $R_v(r,k)$ is the subset of partitions having $v$ as their flagged vertex. Since $|R_v(r,k)|$ is independent of $v$, we have
\begin{equation} \label{eqn:split2}
 \big|\mathscr{R}^{\mathbf{c}}_{a,b}(r,k)\big|=(r+2)|R_v(r,k)|.
\end{equation}

Let $v$ be an arbitrary flagged vertex. Any partition in $R_v(r,k)$ has a distinguished diagonal (containing $v$) which dissects the $(r+2)$-gon into two admissible polygons, one with $\ell-1$ diagonals $(0<\ell<k)$ and the other with $k-\ell$ diagonals. Denote these polygons by $P_\ell$ and $P'_{\ell}$, respectively. As discussed in the introduction, we must have $r=an+(b-2)k$ for some $n\in\mathbb{N}$. With a similar argument, one concludes that $P_\ell$ must have $am+(b-2)\ell+2$ sides (for some $m$), while $P'_{\ell}$ must have $a(n-m)+(b-2)(k-\ell)+2$ sides. Thus the number of colored partitions for $P_\ell$ is $\p(m,\ell)$ and the number for $P'_{\ell}$ is $\p(n-m,k-\ell)$, so there are $\p(n-m,k-\ell)\cdot\p(m,\ell)$ possible colored partitions containing the specified diagonal.

Summing over all possible admissible values of $\ell$ and $m$ (namely, $1\le\ell\le k-1$ and $1\le m\le n-1$) gives
\begin{equation*}
 |R_v(r,k)|=\sum_{\ell=1}^{k-1}\sum_{m=\ell}^{n-1} \p(n-m,k-\ell)\p(m,\ell).
\end{equation*}
Finally, using this identity together with \eqref{eqn:split1} and \eqref{eqn:split2}, we arrive at \eqref{eqn:qrecursion}.  \qed

\section{Proof of Theorem~\ref{thm:qBell}}
\label{sec:pBellProof}
As discussed in the introduction, $\p(n,1)=c_n$ for every $n$, so identity \eqref{eqn:qBell} is satisfied when $k=1$.
To establish the identity for every $k$, we will show that the expression involving the partial Bell polynomials satisfy the same recurrence relation as $\p(n,k)$ for every $n,k\in\mathbb{N}$.

First let us recall the following convolution formula:
\begin{lemma}[{\cite[Corollary~11]{BGW12}}] \label{lem:BellConvolution}
Let $\alpha(\ell,m)$ be a polynomial in $\ell$ and $m$ of degree at most one. For any sequence $x=(x_1,x_2,\dots)$ and any $\tau\in\mathbb{C}$, we have
\begin{equation*}
\sum_{\ell=0}^{k}\sum_{m=\ell}^n \frac{\tau\binom{\alpha(\ell,m)}{k-\ell}\binom{\tau-\alpha(\ell,m)}{\ell}\binom{n}{m}}{\alpha(\ell,m)\big(\tau-\alpha(\ell,m)\big)\binom{k}{\ell}} B_{m,\ell}(x) B_{n-m,k-\ell}(x)
=\tfrac{\tau-\alpha(0,0)+\alpha(k,n)}{\alpha(k,n)(\tau-\alpha(0,0))}\binom{\tau}{k} B_{n,k}(x).
\end{equation*}
\end{lemma}

\medskip
Note that
\begin{equation} \label{eqn:product}
 \frac{\binom{\alpha(\ell,m)}{k-\ell}\binom{\tau-\alpha(\ell,m)}{\ell}
  \binom{n}{m}}{\alpha(\ell,m)\big(\tau-\alpha(\ell,m)\big)\binom{k}{\ell}}
 = \frac{n!}{k!}\left(\frac{\binom{\alpha(\ell,m)}{k-\ell}}{\alpha(\ell,m)}\frac{(k-\ell)!}{(n-m)!}
   \right)\! \left(\frac{\binom{\tau-\alpha(\ell,m)}{\ell}}{\big(\tau-\alpha(\ell,m)\big)}\frac{\ell!}{m!} \right)\!.
\end{equation}
Fix $a$, $b$, $(c_j)$, and let
\begin{equation*}
q(n,k) =\frac{\binom{an+(b-1)k+1}{k}}{an+(b-1)k+1}\frac{k!}{n!}B_{n,k}(1!c_1,2!c_2,\dots).
\end{equation*}
Letting $x_j=j!c_j$, $\tau=an+(b-1)k+2$,  and $\alpha(\ell,m)=a(n-m)+(b-1)(k-\ell)+1$, equation \eqref{eqn:product} gives
\begin{equation*}
 \frac{\binom{\alpha(\ell,m)}{k-\ell}\binom{\tau-\alpha(\ell,m)}{\ell}
 \binom{n}{m}}{\alpha(\ell,m)\big(\tau-\alpha(\ell,m)\big)\binom{k}{\ell}} B_{m,\ell}B_{n-m,k-\ell}
 = \frac{n!}{k!} q(n-m,k-\ell)q(m,\ell),
\end{equation*}
and so by Lemma~\ref{lem:BellConvolution},
\begin{equation*}
\sum_{\ell=0}^{k}\sum_{m=\ell}^n \tau \frac{n!}{k!} q(n-m,k-\ell)q(m,\ell) = 2\binom{\tau}{k} B_{n,k}(x),
\end{equation*}
which implies
\begin{align*}
 \sum_{\ell=0}^{k}\sum_{m=\ell}^n q(n-m,k-\ell)q(m,\ell)
 &= \frac{2}{an+(b-1)k+2}\binom{an+(b-1)k+2}{k} \frac{k!}{n!} B_{n,k}(x) \\
 &= \frac{2}{an+(b-2)k+2}\binom{an+(b-1)k+1}{k} \frac{k!}{n!} B_{n,k}(x) \\
 &= \frac{2(an+(b-1)k+1)}{an+(b-2)k+2}\,q(n,k).
\end{align*}
Therefore,
\begin{align*}
 \sum_{\ell=1}^{k-1}\sum_{m=\ell}^{n-1} q(n-m,k-\ell)q(m,\ell)
 &= \frac{2(an+(b-1)k+1)}{an+(b-2)k+2}\,q(n,k) - 2q(n,k) \\
 &= \frac{2(k-1)}{an+(b-2)k+2}\,q(n,k).
\end{align*}
In other words, $q(n,k)$ satisfies the same recurrence as $\p(n,k)$. Since $\p(n,1)=c_n=q(n,1)$ for every $n$, we conclude $\p(n,k)=q(n,k)$ for every $n$ and $k$. This proves \eqref{eqn:qBell}. \qed

\section{Enumeration of some special partitions}
\label{sec:applications}

In this section we revisit some classical examples of polygon partitions and consider several special choices for the type of polygons allowed in the partitions.

Recall that $\p(n,k)=|\mathscr{P}^{\mathbf{c}}_{a,b}(an+(b-2)k,k)|$ denotes the number of colored partitions of a convex polygon with $an+(b-2)k+2$ sides into $k$ polygons whose number of sides is congruent to $b \pmod a$, and each $(aj+b)$-gon may be colored in $c_j$ ways. In Theorem~\ref{thm:qBell} we gave a representation for $\p(n,k)$, which can be rewritten as
\begin{equation} \label{eqn:qBellb}
\p(n,k) =\frac{1}{k}\binom{an+(b-1)k}{k-1}\frac{k!}{n!}B_{n,k}(1!c_1,2!c_2,\dots).
\end{equation}
\subsection*{Colored polygon partitions}
Let $y_n$ be the total number of colored partitions of a convex $(n+2)$-gon by noncrossing diagonals.  There are $c_1$ possible ways to color a triangle ($n=1$). For $n\ge 2$, we can use identity \eqref{eqn:qBellb} with $a=1$ and $b=2$, together with the identity $\binom{n+k}{k-1}=\frac{k}{n+1} \binom{n+k}{k}$, to obtain
\begin{equation} \label{eqn:totalClassical}
  y_1=c_1,\quad y_n =\frac{1}{n+1} \sum_{k=1}^{n} \binom{n+k}{k}\frac{k!}{n!} B_{n,k}(1!c_1,2!c_2,\dots) \text{ for }n\ge 2.
\end{equation}
If $c_j=1$ for every $j$ (only one color for each type of polygon), then $\frac{k!}{n!}B_{n,k}(1!,2!,\dots) = \binom{n-1}{k-1}$ for $n,k\in\mathbb{N}$ and $(y_n)$ gives the Schr\"oder numbers (A001003 in \cite{Sloane})
\begin{equation*}
  y_n=\frac{1}{n+1} \sum_{k=1}^{n} \binom{n+k}{k}\binom{n-1}{k-1} \text{ for }n\ge 1.
\end{equation*}
\subsection*{Partitions into $(q+2)$-gons}
The total number of partitions into $(q+2)$-gons is obtained from \eqref{eqn:totalClassical} with $c_j=0$ for every $j\not=q$. Clearly, $y_n=0$ unless $n$ is a multiple of $q$, and if $n=qm$, we then get
\begin{align*}
  y_{qm} &=\sum_{k=1}^{qm} \frac{1}{k}\binom{qm+k}{k-1}\frac{k!}{(qm)!}B_{qm,k}(0,\dots,q!c_q,0,\dots)\\
	&=\frac{1}{m}\binom{(q+1)m}{m-1}\frac{m!}{(qm)!}\frac{(qm)!}{m!}c_q^m \\
	&= \frac{1}{qm+1}\binom{(q+1)m}{m}c_q^m.
\end{align*}
If $c_q=1$, these are the Fuss-Catalan numbers.
\subsection*{Partitions into triangles and quadrilaterals}
Assume we are interested in partitions made only by triangles and quadrilaterals. For an $(n+2)$-gon, the total number of such partitions can be obtained from \eqref{eqn:totalClassical} by choosing $c_j=0$ for $j\ge 3$. Thus $y_1=c_1$ and
\begin{equation*}
  y_n = \frac{1}{n+1}\sum_{\frac{n}{2}\le k\le n} \binom{n+k}{k} \binom{k}{n-k} c_1^{2k-n}c_2^{n-k} \text{ for } n\ge 2,
\end{equation*}
since $\frac{k!}{n!} B_{n,k}(1!c_1,2!c_2,0,\dots)=\binom{k}{n-k} c_1^{2k-n}c_2^{n-k}$. If $c_1=c_2=1$, $(y_n)$ is A001002 in \cite{Sloane}.
\subsection*{Partitions into $(q+2)$-gons and $(q+3)$-gons}
The total number of partitions of a convex $(n+2)$-gon into polygons with $q+2$ and $q+3$ sides, $q\in\mathbb{N}$, can be obtained from \eqref{eqn:totalClassical} with the choice $c_j=0$ for every $j\not\in\{q, q+1\}$. For this, let us examine the corresponding partial Bell polynomial. Using formulas [3n] and [3n'] in \cite[Sec.~3.3]{Comtet}, we have
\begin{align*}
  B_{n,k}(0,\dots,\, &q!c_q, (q+1)!c_{q+1}, 0,\dots) \\
  &=\sum_{\kappa\le k,\, \nu\le n} \binom{n}{\nu} B_{\nu,\kappa}(0,\dots, q!c_q, 0,\dots) B_{n-\nu,k-\kappa}(0,\dots, (q+1)!c_{q+1}, 0,\dots),
\end{align*}
which is not zero only if $\nu=q\kappa$ and $n-\nu=(q+1)(k-\kappa)$. Thus we must have $\kappa=(q+1)k-n$ and $k-\kappa=n-qk$, and so
\begin{equation*}
  B_{n,k}(0,\dots,\, q!c_q, (q+1)!c_{q+1}, 0,\dots) =\frac{n!}{k!}\binom{k}{n-qk} c_q^{(q+1)k-n} c_{q+1}^{n-qk}
\end{equation*}
with the usual convention that $\binom{k}{n-qk}=0$ if $k<n-qk$. Then, for $n>q+1$, we arrive at
\begin{equation*}
 y_n =\frac{1}{n+1}\sum_{\frac{n}{q+1}\le k\le \frac{n}{q}} \binom{n+k}{k}\binom{k}{n-qk} 
 c_q^{(q+1)k-n} c_{q+1}^{n-qk}.
\end{equation*}
\subsection*{Partitions into even-gons}
The set of such partitions can be described by choosing $a=2$ and $b=2$ in $\mathscr{P}^{\mathbf{c}}_{a,b}(an+(b-2)k,k)$. More precisely, the total number of partitions of a convex $(2n+2)$-gon made by noncrossing diagonals into polygons with an even number of sides, and such that each $(2j+2)$-gon may be colored in $c_j$ ways, is given by
\begin{equation*}
 y_n = \sum_{k=1}^{n}\frac{1}{k}\binom{2n+k}{k-1}\frac{k!}{n!}B_{n,k}(1!c_1,2!c_2,\dots).
\end{equation*}

In the special case when $c_j=1$ for every $j$, we obtain (cf. A003168 in \cite{Sloane})
\begin{equation*}
 y_n = \frac{1}{n} \sum_{k=1}^{n}\binom{2n+k}{k-1}\binom{n}{k}\; \text{ for } n\ge 1.
\end{equation*}
\subsection*{Partitions into odd-gons}
If we now consider the set $\mathscr{P}^{\mathbf{c}}_{a,b}(an+(b-2)k,k)$ with $a=2$ and $b=1$, the quantity $\p(n,k)=\big|\mathscr{P}^{\mathbf{c}}_{2,1}(2n-k,k)\big|$ gives the number of partitions of a convex polygon with $2n-k+2$ sides by $k-1$ noncrossing diagonals into polygons with an odd number of sides, and such that each $(2j+1)$-gon may be colored in $c_j$ ways. According to \eqref{eqn:qBellb}, we have
\begin{equation*}
\p(n,k) =\frac{1}{k}\binom{2n}{k-1}\frac{k!}{n!}B_{n,k}(1!c_1,2!c_2,\dots).
\end{equation*}
We can then enumerate the total number of such partitions either by the number of edges (exterior sides plus diagonals) or by the number of sides. 

{\em By the number of edges:} Since $\p(n,k)$ refers to a convex polygon with $2n-k+2$ sides and $k-1$ diagonals, there is a total of  $(2n-k+2) + (k-1) = 2n+1$ edges. We now sum over all possible values of $k$ from 2 to $n$ and get
\begin{equation*}
 y_n = \sum_{k=1}^{n}\frac{1}{k} \binom{2n}{k-1}\frac{k!}{n!}B_{n,k}(1!c_1,2!c_2,\dots).
\end{equation*}
If $c_j=1$ for every $j$, then we obtain (cf. A001764 in \cite{Sloane})
\begin{equation*}
 y_n = \frac{1}{n} \sum_{k=1}^{n}\binom{2n}{k-1}\binom{n}{k} = \frac{1}{2n+1}\binom{3n}{n}\; \text{ for } n\ge 1.
\end{equation*}

{\em By the number of sides:} If the polygon has $r+2$ sides, we must have $r+2=2n-k+2$ and so $r+k=2n$. Let $\ell=\frac{r-k}{2}\in\mathbb{N}_0$. Then $k=r-2\ell$, $n=r-\ell$, and
\begin{equation*}
\p(n,k) = \frac{1}{r-2\ell}\binom{2(r-\ell)}{r-2\ell-1}\frac{(r-2\ell)!}{(r-\ell)!}B_{r-\ell,r-2\ell}(1!c_1,2!c_2,\dots).
\end{equation*}
Summing over all possible values of $\ell$ we arrive at
\begin{equation*}
 y_r = \sum_{\ell=0}^{\lfloor\frac{r-1}{2}\rfloor} \frac{1}{r-2\ell}\binom{2(r-\ell)}{r-2\ell-1}\frac{(r-2\ell)!}{(r-\ell)!}B_{r-\ell,r-2\ell}(1!c_1,2!c_2,\dots).
\end{equation*}
If $c_j=1$ for every $j$, then we obtain (cf. A049124 in \cite{Sloane})
\begin{equation*}
 y_r = \frac{1}{r+1}\sum_{\ell=0}^{\lfloor\frac{r-1}{2}\rfloor} \binom{2(r-\ell)}{r}\binom{r-\ell-1}{\ell} \; \text{ for } r\ge 1.
\end{equation*}

\subsection*{Triangle free partitions}
In order to count all polygon partitions that avoid triangles, we just need to set $c_1=0$ in \eqref{eqn:totalClassical}. Then, for $n\ge 2$,
\begin{align*}
  y_n &=\frac{1}{n+1} \sum_{k=1}^{n} \binom{n+k}{k}\frac{k!}{n!} B_{n,k}(0,2!c_2,3!c_3\dots) \\
  &=\frac{1}{n+1} \sum_{k=1}^{n-1} \binom{n+k}{k}\frac{k!}{(n-k)!} B_{n-k,k}(1!c_2,2!c_3\dots),
\end{align*}
since $B_{n,k}(0,2!c_2,3!c_3\dots)=\frac{n!}{(n-k)!}B_{n-k,k}(1!c_2,2!c_3\dots)$ by identity [3l'] in \cite[Sec.~3.3]{Comtet}.

In the case when $c_j=1$ for every $j\ge 2$, this sequence reduces to
\begin{equation*}
  y_1=0, \quad y_n =\frac{1}{n+1} \sum_{k=1}^{n-1} \binom{n+k}{k}\binom{n-k-1}{k-1} \text{ for } n\ge 2.
\end{equation*}
If we let $y_0=1$, this gives A046736 in \cite{Sloane}.
\subsection*{Partitions into $(j+2)$-gons with $j>q$ }
In this more general situation, we are looking at the total number of partitions that avoid $3$-gons, $4$-gons, and so on up to $(q+2)$-gons.  This number can be easily obtained from \eqref{eqn:totalClassical} by setting $c_1=c_2=\dots=c_q=0$. Repeated use of identity [3l'] in \cite[Section~3.3]{Comtet} gives, for $n\ge q+1$,
\begin{align*}
  y_n &=\frac{1}{n+1}\sum_{k=1}^{n} \binom{n+k}{k}\frac{k!}{n!} B_{n,k}(0,\dots,(q+1)!c_{q+1},(q+2)!c_{q+2},\dots) \\
  &=\frac{1}{n+1}\sum_{1\le k< n/q} \binom{n+k}{k}\frac{k!}{(n-qk)!} B_{n-qk,k}(1!c_{q+1},2!c_{q+2}\dots).
\end{align*}

Now, if we ignore the coloring and set $c_j=1$ for every $j\ge q+1$, we arrive at the sequence
\begin{gather*}
  y_1=\dots=y_{q}=0, \\
  y_n =\frac{1}{n+1}\sum_{1\le k <n/q} \binom{n+k}{k}\binom{n-q k-1}{k-1} \text{ for } n\ge q+1.
\end{gather*}
The special cases when $q=2$ or $q=3$ lead to the sequences A054514 and A215342 in \cite{Sloane}.
\subsection*{Partitions avoiding $(q+2)$-gons}
We conclude our illustrating examples by considering colored partitions that avoid polygons with $q+2$ sides for some fixed $q\in\mathbb{N}$. Given a convex $(n+2)$-gon, the number $y_n$ of all such partitions can be obtained from \eqref{eqn:totalClassical} simply by choosing $c_q=0$. That is,
\begin{equation*}
  y_n =\frac{1}{n+1} \sum_{k=1}^{n} \binom{n+k}{k}\frac{k!}{n!} 
  B_{n,k}(1!c_1,\dots,(q-1)!c_{q-1},0,(q+1)!c_{q+1},\dots). 
\end{equation*}
Let $\bar x_q$, $\bar e_q$, and $\bar x$ be the sequences defined as $\bar x_q=(1!c_1,\dots,(q-1)!c_{q-1},0,(q+1)!c_{q+1},\dots)$, $\bar e_q=(0,\dots,0,q!,0,\dots)$, and $\bar x = \bar e_q + \bar x_q$. Using identities [3n] and [3n'] in \cite[Sec.~3.3]{Comtet}, we then have
\begin{equation*}
 B_{n,k}(\bar x_q) = \!\sum_{\ell\le k,\, \nu\le n} \binom{n}{\nu} B_{\nu,\ell}(-\bar e_q)B_{n-\nu,k-\ell}(\bar x) 
 = \sum_{\ell\le k} \binom{n}{q\ell}\frac{(q\ell)!}{\ell!} (-1)^\ell B_{n-q\ell,k-\ell}(\bar x),
\end{equation*}
and so
\begin{equation}\label{eq:qavoid}
 y_n =\frac{1}{n+1} \sum_{k=1}^{n} \;
 \sum_{\ell\le k} (-1)^\ell \binom{n+k}{k}\binom{k}{\ell}\frac{(k-\ell)!}{(n-q\ell)!} B_{n-q\ell,k-\ell}(\bar x),
\end{equation} 
with the standard convention that $B_{\nu,\kappa}=0$ unless $\nu\ge \kappa\ge 0$, and $B_{\nu,0}=0$ for $\nu>0$.

If $c_j=1$ for every $j$ in $\bar x_q$, then $\bar x=(1!,2!,3!,\dots)$ and $\frac{(k-\ell)!}{(n-q\ell)!} B_{n-q\ell,k-\ell}(\bar x)=\binom{n-q\ell-1}{k-\ell-1}$ for $k-\ell\ge 1$ and $n-q\ell\ge 1$. In this case, the term $\ell=0$ in \eqref{eq:qavoid} gives the Schr\"oder numbers 
\[ s_n = \frac{1}{n+1} \sum_{k=1}^{n} \binom{n+k}{k}\binom{n-1}{k-1} \]
and $y_n$ can be rewritten as
\begin{align*}
 y_n &= s_n + \frac{1}{n+1} \sum_{k=1}^{n} \sum_{\ell=1}^k (-1)^\ell 
 \binom{n+k}{k}\binom{k}{\ell}\frac{(k-\ell)!}{(n-q\ell)!} B_{n-q\ell,k-\ell}(\bar x) \\
 &=s_n + \frac{1}{n+1} \sum_{\ell=1}^{\lfloor n/q\rfloor} \sum_{k=\ell}^n (-1)^\ell 
 \binom{n+k}{k}\binom{k}{\ell}\frac{(k-\ell)!}{(n-q\ell)!} B_{n-q\ell,k-\ell}(\bar x).
\end{align*} 
If $n=qm$, then
\begin{equation*}
 y_{qm} = \frac{(-1)^{m}}{qm+1} \binom{(q+1)m}{m} + \frac{1}{qm+1} \sum_{\ell=0}^{m-1}\; 
 \sum_{k=\ell+1}^{qm-q\ell+\ell} (-1)^\ell \binom{qm+k}{k}\binom{k}{\ell}\binom{qm-q\ell-1}{k-\ell-1},
\end{equation*} 
and if $n\not\equiv 0$ modulo $q$, then
\begin{equation*}
 y_{n} = \frac{1}{n+1} \sum_{\ell=0}^{\lfloor n/q\rfloor}\; \sum_{k=\ell+1}^{n-q\ell+\ell} (-1)^\ell 
 \binom{n+k}{k}\binom{k}{\ell}\binom{n-q\ell-1}{k-\ell-1}.
\end{equation*} 

In the special case when $q=2$ (partitions avoiding quadrilaterals, cf. A054515 in \cite{Sloane}), we get the explicit formula
\begin{equation*}
 y_{n} = (-1)^{\lfloor n/2\rfloor}t_n + \frac{1}{n+1} \sum_{\ell=0}^{\lfloor n/2\rfloor-1}\; \sum_{k=\ell+1}^{n-\ell} (-1)^\ell 
 \binom{n+k}{k}\binom{k}{\ell}\binom{n-2\ell-1}{k-\ell-1},
\end{equation*} 
where $t_{2m}=\frac{1}{2m+1} \binom{3m}{m}$ and $t_{2m+1}= \binom{3m+2}{m}$. 

\medskip
Cases beyond $q=2$ give interesting sequences currently not listed in the OEIS.

\section{Further examples and generalizations}
\label{sec:related}

In this section we will discuss a generalization of the following problem: 
\begin{quote}
Finding the number of partitions of a convex polygon (made by noncrossing diagonals) that have a triangle over a fixed side of the polygon. 
\end{quote}
For example, in the case of a pentagon, there are 7 possible such partitions:

\smallskip
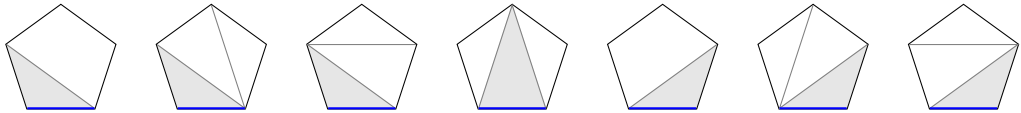
\begin{center}
\begin{tikzpicture}
\node[regular polygon, regular polygon sides=5, minimum size=1.5cm] at (0,0) (A) {};
\draw[gray!20,fill] (A.corner 3) -- (A.corner 4) -- (A.corner 2) -- cycle;
\node[regular polygon, regular polygon sides=5, minimum size=1.54cm, draw] at (0,0) (B) {};
\draw[gray] (A.corner 4) -- (A.corner 2);
\draw[thick,blue] (A.corner 3) -- (A.corner 4);
\node[regular polygon, regular polygon sides=5, minimum size=1.5cm] at (2,0) (A) {};
\draw[gray!20,fill] (A.corner 3) -- (A.corner 4) -- (A.corner 2) -- cycle;
\node[regular polygon, regular polygon sides=5, minimum size=1.54cm, draw] at (2,0) (B) {};
\draw[gray] (A.corner 1) -- (A.corner 4) -- (A.corner 2);
\draw[thick,blue] (A.corner 3) -- (A.corner 4);
\node[regular polygon, regular polygon sides=5, minimum size=1.5cm] at (4,0) (A) {};
\draw[gray!20,fill] (A.corner 3) -- (A.corner 4) -- (A.corner 2) -- cycle;
\node[regular polygon, regular polygon sides=5, minimum size=1.54cm, draw] at (4,0) (B) {};
\draw[gray] (A.corner 4) -- (A.corner 2) -- (A.corner 5);
\draw[thick,blue] (A.corner 3) -- (A.corner 4);
\node[regular polygon, regular polygon sides=5, minimum size=1.5cm] at (6,0) (A) {};
\draw[gray!20,fill] (A.corner 3) -- (A.corner 4) -- (A.corner 1) -- cycle;
\node[regular polygon, regular polygon sides=5, minimum size=1.54cm, draw] at (6,0) (B) {};
\draw[gray] (A.corner 3) -- (A.corner 4) -- (A.corner 1) -- cycle;
\draw[thick,blue] (A.corner 3) -- (A.corner 4);
\node[regular polygon, regular polygon sides=5, minimum size=1.5cm] at (8,0) (A) {};
\draw[gray!20,fill] (A.corner 3) -- (A.corner 4) -- (A.corner 5) -- cycle;
\node[regular polygon, regular polygon sides=5, minimum size=1.54cm, draw] at (8,0) (B) {};
\draw[gray] (A.corner 3) -- (A.corner 5);
\draw[thick,blue] (A.corner 3) -- (A.corner 4);
\node[regular polygon, regular polygon sides=5, minimum size=1.5cm] at (10,0) (A) {};
\draw[gray!20,fill] (A.corner 3) -- (A.corner 4) -- (A.corner 5) -- cycle;
\node[regular polygon, regular polygon sides=5, minimum size=1.54cm, draw] at (10,0) (B) {};
\draw[gray] (A.corner 1) -- (A.corner 3) -- (A.corner 5);
\draw[thick,blue] (A.corner 3) -- (A.corner 4);
\node[regular polygon, regular polygon sides=5, minimum size=1.5cm] at (12,0) (A) {};
\draw[gray!20,fill] (A.corner 3) -- (A.corner 4) -- (A.corner 5) -- cycle;
\node[regular polygon, regular polygon sides=5, minimum size=1.54cm, draw] at (12,0) (B) {};
\draw[gray] (A.corner 2) -- (A.corner 5) -- (A.corner 3);
\draw[thick,blue] (A.corner 3) -- (A.corner 4);
\end{tikzpicture}
\end{center}
\smallskip

Let $y_n$ be the number of such partitions for a convex $(n+3)$-gon. 
Note that having a triangle over a fixed side (base) splits the polygon into two or three parts. If the triangle over the fixed side has another side on the polygon, the second part is a convex $(n+2)$-gon. Otherwise, there are two polygons around the triangle with less than $n+2$ sides each. 

\begin{figure}[h] 
\begin{tikzpicture}
\node[regular polygon, regular polygon sides=9, minimum size=3cm] at (0,0) (A) {};
\draw[gray!20,fill] (A.corner 4) -- (A.corner 5) -- (A.corner 6) -- cycle;
\node[regular polygon, regular polygon sides=9, minimum size=3.04cm, draw] at (0,0) (B) {};
\draw[gray] (A.corner 6) -- (A.corner 4);
\draw[thick,blue] (A.corner 5) -- (A.corner 6);
\node[below] at (0,-1.7) {\small (i) Splitting into two parts};
\node[regular polygon, regular polygon sides=9, minimum size=3cm] at (5,0) (A) {};
\draw[gray!20,fill] (A.corner 5) -- (A.corner 6) -- (A.corner 2) -- cycle;
\node[regular polygon, regular polygon sides=9, minimum size=3.05cm, draw] at (5,0) (B) {};
\draw[gray] (A.corner 5) -- (A.corner 6) -- (A.corner 2) -- cycle;
\draw[thick,blue] (A.corner 5) -- (A.corner 6);
\node[below] at (5,-1.7) {\small (ii) Splitting into three parts};
\end{tikzpicture}
\end{figure}

In both cases, the number of (colored) partitions of each of the polygons involved in this initial splitting can be found by means of Theorems~\ref{thm:qrecursion} and \ref{thm:qBell} with $a=1$ and $b=2$. With
\begin{equation}\label{eqn:special_phat}
\p(n,k) =\frac{\binom{n+k+1}{k}}{(n+k+1)}\frac{k!}{n!}B_{n,k}(1!c_1,2!c_2,\dots),
\end{equation}
the number of sought-after partitions with $k+1$ parts (including the triangle) is given by
\begin{equation*}
 \sum_{\ell=0}^{k}\sum_{m=\ell}^n \p(n-m,k-\ell)\p(m,\ell) = \frac{2}{n+2}\,
 \binom{n+k+1}{k}\frac{k!}{n!}B_{n,k}(1!c_1,2!c_2,\dots),
\end{equation*}
where $\ell=0$ corresponds to the case (i) in the above figure. Letting $c_j=1$ for every $j$ (only one color) and adding over all possible values of $k$, we arrive at
\begin{equation*}
y_n = \frac{2}{n+2}\sum_{k=1}^n\binom{n+k+1}{k}\binom{n-1}{k-1} \text{ for } n\ge 1.
\end{equation*}
This gives the sequence $1,2,7,28,121,550,\dots$, which is listed as A010683 in \cite{Sloane}.

\subsection*{Generalization}
Assume now that we request to have a $(d+1)$-gon over a fixed side of a convex $(n+d+1)$-gon. In this case, the $(d+1)$-gon splits the rest of the polygon into at most $d$ parts, each of which is a convex $(m_i+2)$-gon that may be dissected into $\ell_i$ parts ($\ell_i$ possibly being zero) so that the total number of parts, not including the $(d+1)$-gon, is $k$. 
\begin{figure}[h] 
\begin{tikzpicture}
\node[regular polygon, regular polygon sides=12, minimum size=3cm] at (0,0) (A) {};
\draw[gray!15,fill] (A.corner 7) -- (A.corner 8) -- (A.corner 10) -- (A.corner 1) -- (A.corner 4) -- cycle;
\node[regular polygon, regular polygon sides=12, minimum size=3.04cm, draw] at (0,0) (B) {};
\draw[gray] (A.corner 8) -- (A.corner 10) -- (A.corner 1) -- (A.corner 4) -- (A.corner 7);
\draw[thick,blue] (A.corner 7) -- (A.corner 8);
\node[below] at (0,-1.7) {\small Example: Dodecagon with a pentagon over a fixed base};
\end{tikzpicture}
\end{figure}
Thus $\ell_1+\dots+\ell_d=k$ and $m_1+\dots+m_d=n$. Since each $(m_i+2)$-gon may be partitioned into $\ell_i$ parts in $\p(m_i,\ell_i)$ ways, we need the following more general formula. 

\begin{proposition} \label{prop:d-convolution}
For $\p(n,k)$ as in \eqref{eqn:special_phat}, we have
\begin{equation*}
 \sum_{\substack{\ell_1+\dots+\ell_d=k\\m_1+\dots+m_d=n}} \p(m_1,\ell_1)\dots \p(m_d,\ell_d)
 = \frac{d\binom{n+k+d}{k}}{(n+k+d)} \frac{k!}{n!}B_{n,k}(1!c_1,2!c_2,\dots).
\end{equation*}
\end{proposition}
\begin{proof} 
We will prove this by induction on $d$. When $d=1$ the formula reduces to a tautology. For $d=2$ the formula was proven as part of the proof of Theorem~\ref{thm:qBell}.  For any $d\ge 2$, let us assume that the statement is true for $d-1$.  Then
\begin{align*}
 \sum_{\substack{\ell_1+\dots+\ell_d=k\\m_1+\dots+m_d=n}} &\p(m_1,\ell_1)\cdots \p(m_d,\ell_d) \\
 &= \sum_{\ell=0}^k \sum_{m=0}^n \p(m,\ell) \!\!\! \sum_{\substack{\ell_1+\dots+\ell_{d-1}=k-\ell\\m_1+\dots+m_{d-1}=n-m}} \! \p(m_1,\ell_1)\cdots \p(m_{d-1},\ell_{d-1})\\
&=\frac{(d-1) k!}{n!}\sum_{\ell=0}^k\sum_{m=0}^n \frac{\binom{m+\ell+1}{\ell}}{(m+\ell+1)}
\frac{\binom{n-m+k-\ell+d-1}{k-\ell}}{(n-m+k-\ell+d-1)}\frac{\binom{n}{m}}{\binom{k}{\ell}}B_{m,\ell}B_{n-m,k-\ell} \\
&=\frac{d\binom{n+k+d}{k}}{(n+k+d)}\frac{k!}{n!}B_{n,k}(1!c_1,2!c_2,\dots),
\end{align*}
using Lemma~\ref{lem:BellConvolution} with $\tau=n+k+d$ and $\alpha=n-m+k-\ell+d-1$.
\end{proof}

\begin{corollary}
The total number of partitions $y_n$ of a convex $(n+d+1)$-gon having a $(d+1)$-gon over a fixed side is given by
\begin{equation*}
y_n = \frac{d}{n+d}\sum_{k=1}^n\binom{n+k+d-1}{k}\binom{n-1}{k-1} \text{ for } n\ge 1.
\end{equation*}
\end{corollary}

For cases where the partitions may be colored and there are restrictions on the number of sides of the polygons (along the lines of the problems discussed in previous sections), it is useful to have the following straightforward generalization of Proposition~\ref{prop:d-convolution}.

\begin{proposition}
For $\p(n,k)$ defined by
\begin{equation*}
\p(n,k) =\frac{\binom{an+(b-1)k+1}{k}}{(n+(b-1)k+1)}\frac{k!}{n!}B_{n,k}(1!c_1,2!c_2,\dots),
\end{equation*}
we have 
\begin{equation*}
 \sum_{\substack{\ell_1+\dots+\ell_d=k\\m_1+\dots+m_d=n}} \p(m_1,\ell_1)\dots \p(m_d,\ell_d)
 = \frac{d\binom{an+(b-1)k+d}{k}}{(an+(b-1)k+d)} \frac{k!}{n!}B_{n,k}(1!c_1,2!c_2,\dots).
\end{equation*}
\end{proposition}

\medskip
We conclude the paper with a related for the interested reader.
\subsection*{Problem}
For a convex polygon, find the total number of partitions made by noncrossing diagonals that (i) contain exactly one triangle, (ii) contain exactly one $(d+1)$-gon.


\end{document}